\documentclass[amstex,12pt, amssymb]{article}

\usepackage{mathtext}
\usepackage[cp1251]{inputenc}
\usepackage[T2A]{fontenc}
\usepackage[dvips]{graphicx}
\usepackage{amsmath}
\usepackage{amssymb}
\usepackage{amsxtra}
\usepackage{latexsym}
\usepackage{ifthen}

\newcounter{lemma}[section]

\newcounter{corollary}[section]

\newcounter{remark}[section]

\newcounter{theorem}[section]

\newcounter{proposition}[section]

\newcounter{example}

\numberwithin{equation}{section}

\begin{document}

\markboth{V. DESYATKA, E.~SEVOST'YANOV}{\centerline{ON THE BOUNDARY
BEHAVIOR ...}}

\def\cc{\setcounter{equation}{0}
\setcounter{figure}{0}\setcounter{table}{0}}

\overfullrule=0pt

%\normalsize\large

\author{VICTORIA DESYATKA, EVGENY SEVOST'YANOV}

\title{
{\bf ON UNCLOSED MAPPINGS WITH POLETSKY INEQUALITY}}

\date{\today}
\maketitle

%\large
\begin{abstract}
The manuscript is devoted to the boundary behavior of mappings with
bounded and finite distortion. We consider mappings of domains of
the Euclidean space that satisfy weighted Poletsky inequality.
Assume that, the definition domain is finitely connected on its
boundary and, in addition, on the set of all points which are
pre-images of the cluster set of this boundary. Then the specified
mappings have a continuous boundary extension provided that the
majorant in the Poletsky inequality satisfies some integral
divergence condition, or has a finite mean oscillation at every
boundary point.
\end{abstract}


\section{Introduction}

Results on continuous boundary extension of quasiconformal mappings
and their generalizations are well known. On this topic we could
point to the following publications, see e.g. \cite{Cr},
\cite{GRSY}, \cite{MRSY$_2$}, \cite{Na$_1$}--\cite{Na$_2$} and
\cite{Vu}. We should also note a number of publications that assume
that the mapping is closed. The condition for the mapping to be
closed is very important when solving this problem, see, for
example, \cite[Theorem~3.3]{Vu} and \cite{SSD}. This article is
devoted to studying the case when the mapping is not closed, and the
closeness condition itself is replaced by some other, milder
conditions in comparison with it.

\medskip
From this position, it should be said that the (conditional)
majority of known mappings, including analytic functions, are not
closed, and the study of their boundary behavior may turn out to be
more important than the study of homeomorphisms or (more general)
closed mappings. For example, an analytic function $f(z)=z^4,$ on
the plane is a closed mapping in the unit disk ${\Bbb D}=\{z\in
{\Bbb C}: |z|<1\}$ but in the disk $B(1, 1)=\{z\in {\Bbb C}:
|z-1|<1\}$ it is no longer such. Due to the well-known theorems of
Casorati-Weierstrass or Picard, even an isolated singularity of an
analytic function will be its essential singularity only in the case
when the mapping is not closed (closed mappings, in particular,
preserve the boundary of the domain, and, at the same time, the
cluster set of an analytic function with an essentially singularity
is the entire extended complex plane).

\medskip
Subsequent studies are related to the concept of the modulus of
families of paths, the definition of which is given below. A Borel
function $\rho:{\Bbb R}^n\,\rightarrow [0,\infty] $ is called {\it
admissible} for the family $\Gamma$ of paths $\gamma$ in ${\Bbb
R}^n,$ if the relation
\begin{equation}\label{eq1.4}
\int\limits_{\gamma}\rho (x)\, |dx|\geqslant 1
\end{equation}
holds for all (locally rectifiable) paths $ \gamma \in \Gamma.$ In
this case, we write: $\rho \in {\rm adm} \,\Gamma .$ Let $p\geqslant
1,$ then {\it $p$-modulus} of $\Gamma $ is defined by the equality
\begin{equation}\label{eq1.3gl0}
M_p(\Gamma)=\inf\limits_{\rho \in \,{\rm adm}\,\Gamma}
\int\limits_{{\Bbb R}^n} \rho^p (x)\,dm(x)\,.
\end{equation}
We set $M(\Gamma):=M_n(\Gamma).$ The Poletsky classical inequality
is the inequality
\begin{equation}\label{eq5}
M(f(\Gamma))\leqslant K\cdot M(\Gamma)\,,
\end{equation}
where $K\geqslant 1$ is some constant, and $\Gamma$ is an arbitrary
family of paths in $D.$ It was obtained by E. Poletsky in the early
70s last century, see~\cite[Theorem~1]{Pol}. Our further research is
related to the fact that we are considering mappings of a more
general nature, see below. Let $x_0\in {\Bbb R}^n,$
$0<r_1<r_2<\infty,$
\begin{equation}\label{eq1ED}
S(x_0,r) = \{ x\,\in\,{\Bbb R}^n : |x-x_0|=r\}\end{equation}
and
\begin{equation}\label{eq1**}
A=A(x_0, r_1,r_2)=\left\{ y\,\in\,{\Bbb R}^n:
r_1<|y-y_0|<r_2\right\}\,.\end{equation}
Given sets $E,$ $F\subset\overline{{\Bbb R}^n}$ and a domain
$D\subset {\Bbb R}^n$ we denote by $\Gamma(E,F,D)$ a family of all
paths $\gamma:[a,b]\rightarrow \overline{{\Bbb R}^n}$ such that
$\gamma(a)\in E,\gamma(b)\in\,F $ and $\gamma(t)\in D$ for $t \in
(a, b).$ Let $S_i=S(x_0, r_i),$ $i=1,2,$ where spheres $S(x_0, r_i)$
centered at $x_0$ of the radius $r_i$ are defined in~(\ref{eq1ED}).
Let $Q:{\Bbb R}^n\rightarrow {\Bbb R}^n$ be a Lebesgue measurable
function satisfying the condition $Q(x)\equiv 0$ for $x\in{\Bbb
R}^n\setminus D.$ Let $p\geqslant 1.$ A mapping $f:D\rightarrow
\overline{{\Bbb R}^n}$ is called a {\it ring $Q$-mapping at the
point $x_0\in \overline{D}\setminus \{\infty\}$ with respect to
$p$-modulus}, if the condition
\begin{equation} \label{eq2*!A}
M_p(f(\Gamma(S_1, S_2, D)))\leqslant \int\limits_{A\cap D} Q(x)\cdot
\eta^n (|x-x_0|)\, dm(x)
\end{equation}
holds for all $0<r_1<r_2<d_0:=\sup\limits_{x\in D}|x-x_0|$ and all
Lebesgue measurable functions $\eta:(r_1, r_2)\rightarrow [0,
\infty]$ such that
\begin{equation}\label{eqA2}
\int\limits_{r_1}^{r_2}\eta(r)\,dr\geqslant 1\,.
\end{equation}
Note that, (\ref{eq5}) implies~(\ref{eq2*!A})--(\ref{eqA2}) whenever
$p=n$ and $Q(x)\equiv K.$ So, (\ref{eq2*!A}) may be interpreted as
Poletsky weighed inequality with the weight $Q(x).$

\begin{remark}\label{rem1}
Note that all quasiregular mappings $f:D\rightarrow {\Bbb R}^n$
satisfy the condition
\begin{equation}\label{eq22}
M(f(\Gamma(S_1, S_2, D)))\leqslant \int\limits_{D\cap
A(y_0,r_1,r_2)} K_I\cdot \eta^n (|x-x_0|)\, dm(x) \end{equation}
at each point $x_0\in \overline{D}\setminus\{\infty\}$ with some
constant $K_I=K_I(f)\geqslant 1$ and an arbitrary
Lebesgue-dimensional function $\eta: (r_1,r_2)\rightarrow
[0,\infty],$ which satisfies condition~(\ref{eqA2}). Indeed,
quasiregular mappings satisfy the condition
\begin{equation}\label{eq24}
M(f(\Gamma(S_1, S_2, D)))\leqslant \int\limits_{D\cap
A(x_0,r_1,r_2)} K_I\cdot \rho^n(x)\, dm(x)
\end{equation}
for an arbitrary function $\rho\in{\rm adm}\, \Gamma(S_1, S_2, D)),$
see \cite[Theorem~8.1.II]{Ri}. Put $\rho(x):=\eta(|x-x_0|)$ for
$x\in A(x_0,r_1,r_2)\cap D,$ and $\rho(x)=0$ otherwise. By Luzin
theorem, we may assume that the function $\rho$ is Borel measurable
(see, e.g., \cite[Section~2.3.6]{Fe}). Due
to~\cite[Theorem~5.7]{Va},
$$\int\limits_{\gamma}\rho(x)\,|dx|\geqslant
\int\limits_{r_1}^{r_2}\eta(r)\,dr\geqslant 1$$
for each (locally rectifiable) path $\gamma$ in $\Gamma(S(x_0, r_1),
S(x_0, r_2), A(x_0, r_1, r_2)).$ By substituting the function $\rho$
mentioned above into~(\ref{eq24}), we obtain the desired
ratio~(\ref{eq22}).
\end{remark}

\medskip
Recall that a mapping $f:D\rightarrow {\Bbb R}^n$ is called {\it
discrete} if the pre-image $\{f^{-1}\left(y\right)\}$ of each point
$y\,\in\,{\Bbb R}^n$ consists of isolated points, and {\it is open}
if the image of any open set $U\subset D$ is an open set in ${\Bbb
R}^n.$ Later, in the extended space $\overline{{{\Bbb R}}^n}={{\Bbb
R}}^n\cup\{\infty\}$ we use the {\it spherical (chordal) metric}
$h(x,y)=|\pi(x)-\pi(y)|,$ where $\pi$ is a stereographic projection
$\overline{{{\Bbb R}}^n}$ onto the sphere
$S^n(\frac{1}{2}e_{n+1},\frac{1}{2})$ in ${{\Bbb R}}^{n+1},$ namely,
$$h(x,\infty)=\frac{1}{\sqrt{1+{|x|}^2}}\,,$$
\begin{equation}\label{eq3C}
h(x,y)=\frac{|x-y|}{\sqrt{1+{|x|}^2} \sqrt{1+{|y|}^2}}\,, \quad x\ne
\infty\ne y
\end{equation}
(see \cite[Definition~12.1]{Va}). Further, the closure
$\overline{A}$ and the boundary $\partial A$ of the set $A\subset
\overline{{\Bbb R}^n}$ we understand relative to the chordal metric
$h$ in $\overline{{\Bbb R}^n}.$

\medskip
The boundary of $D$ is called {\it weakly flat} at the point $x_0\in
\partial D,$ if for every $P>0$ and for any neighborhood $U$
of the point $x_0$ there is a neighborhood $V\subset U$ of the same
point such that $M(\Gamma(E, F, D))>P$ for any continua $E, F\subset
D$ such that $E\cap\partial U\ne\varnothing\ne E\cap\ partial V$ and
$F\cap\partial U\ne\varnothing\ne F\cap\partial V.$ The boundary of
$D$ is called weakly flat if the corresponding property holds at any
point of the boundary $D.$

\medskip
Given a mapping $f:D\rightarrow {\Bbb R}^n$, we denote
\begin{equation}\label{eq1_A_4} C(f, x):=\{y\in \overline{{\Bbb
R}^n}:\exists\,x_k\in D: x_k\rightarrow x, f(x_k) \rightarrow y,
k\rightarrow\infty\}
\end{equation}
and
\begin{equation}\label{eq1_A_5} C(f, \partial
D)=\bigcup\limits_{x\in \partial D}C(f, x)\,.
\end{equation}
In what follows, ${\rm Int\,}A$ denotes the set of inner points of
the set $A\subset \overline{{\Bbb R}^n}.$ Recall that the set
$U\subset\overline{{\Bbb R}^n}$ is neighborhood of the point $z_0,$
if $z_0\in {\rm Int\,}A.$

\medskip
We say that a function ${\varphi}:D\rightarrow{\Bbb R}$ has a {\it
finite mean oscillation} at a point $x_0\in D,$ write $\varphi\in
FMO(x_0),$ if
\begin{equation}\label{eq29*!}
\limsup\limits_{\varepsilon\rightarrow
0}\frac{1}{\Omega_n\varepsilon^n}\int\limits_{B( x_0,\,\varepsilon)}
|{\varphi}(x)-\overline{{\varphi}}_{\varepsilon}|\ dm(x)<\infty\,,
\end{equation}
where
$$\overline{{\varphi}}_{\varepsilon}=\frac{1}
{\Omega_n\varepsilon^n}\int\limits_{B(x_0,\,\varepsilon)}
{\varphi}(x) dm(x)\,.$$
Let $Q:{\Bbb R}^n\rightarrow [0,\infty]$ be a Lebesgue measurable
function. We set
$$Q^{\,\prime}(x)=\left\{
\begin{array}{rr}
Q(x), &   Q(x)\geqslant 1\,, \\
1,  &  Q(x)<1\,.
\end{array}
\right.$$ Denote by $q^{\,\prime}_{x_0}$ the mean value of
$Q^{\,\prime}(x)$ over the sphere $|x-x_0|=r$, that means,
\begin{equation}\label{eq32*B}
q^{\,\prime}_{x_0}(r):=\frac{1}{\omega_{n-1}r^{n-1}}
\int\limits_{|x-x_0|=r}Q^{\,\prime}(x)\,d{\mathcal H}^{n-1}\,.
\end{equation}
Note that, using the inversion $\psi(x)=\frac{x}{|x|^2},$ we may
give the definition of $FMO$ as well as the quantity
in~(\ref{eq32*B}) for $x_0=\infty.$

\medskip
We say that the boundary $\partial D$ of a~domain $D$ in ${\Bbb
R}^n,$ $n\geqslant 2,$ is {\it strongly accessible at a~point
$x_0\in
\partial D$ with respect to the $p$-modulus} if for each neighborhood
$U$ of $x_0$ there
exist a~compact set $E\subset D$, a~neighborhood $V\subset U$ of
$x_0$ and $\delta>0$ such that
\begin{equation}
\label{eq1.3_a} M_p(\Gamma(E,F, D))\geqslant \delta
\end{equation}
for each continuum $F$ in~$D$ that intersects $\partial U$ and
$\partial V$. When $p=n$, we will usually drop the prefix in the
''$p$-modulus'' when speaking about~(\ref{eq1.3_a}).

\medskip
Some analogues of the following result were established for the case
of homeomorphisms in \cite[Lemma~5.20, Corollary~5.23]{MRSY$_1$},
\cite[Lemma~6.1, Theorem~6.1]{RS} and \cite[Lemma~5, Theorem~3]{Sm}.
For open discrete and closed mappings, see, e.g., in \cite{SSD},
\cite{Sev$_1$} and \cite{Sev$_2$}.

\medskip
\begin{theorem}\label{th3}
{\it\, Let $p\geqslant 1,$ let $D$ and $D^{\,\prime}$ be domains in
${\Bbb R}^n,$ $n\geqslant 2,$ $f:D\rightarrow D^{\,\prime}$ be an
open discrete mapping satisfying
relations~(\ref{eq2*!A})--(\ref{eqA2}) at the point $b\in
\partial D$ with respect to $p$-modulus,
$f(D)=D^{\,\prime}.$ In addition, assume that

\medskip
1) the set
$$E:=f^{\,-1}(C(f, \partial
D))$$
is nowhere dense in $D$ and $D$ is finitely connected on $E,$ i.e.,
for any $z_0\in E$ and any neighborhood $\widetilde{U}$ of $z_0$
there is a neighborhood $\widetilde{V}\subset \widetilde{U}$ of
$z_0$ such that $(D\cap \widetilde{V})\setminus E$ consists of
finite number of components.

\medskip
2) for any neighborhood $U$ of $b$ there is a neighborhood $V\subset
U$ of $b$ such that:

\medskip
2a) $V\cap D$ is connected,

\medskip
2b) $(V\cap D)\setminus E$ consists at most of $m$ components,
$1\leqslant m<\infty,$

\medskip
3) $D^{\,\prime}\setminus C(f, \partial D)$ consists of finite
components, each of them has a strongly accessible boundary with
respect to $p$-modulus.

Suppose that at least one of the following conditions is satisfied:
1) a function $Q$ has a finite mean oscillation at a point $b;$ 2)
$q_{b}(r)\,=\,O\left(\left[\log{\frac1r}\right]^{n-1}\right)$ as
$r\rightarrow 0;$ 3) the condition
\begin{equation}\label{eq6}
\int\limits_{0}^{\delta(b)}\frac{dt}{t^{\frac{n-1}{p-1}}
q_{b}^{\,\prime\,\frac{1}{p-1}}(t)}=\infty
\end{equation}
holds for some $\delta(b)>0.$ Then $f$ has a continuous extension to
$b.$

If the above is true for any point $b\in\partial D,$ the mapping $f$
has a continuous extension
$\overline{f}:\overline{D}\rightarrow\overline{D^{\,\prime}},$
moreover, $\overline{f}(\overline{D})=\overline{D^{\,\prime}}.$ }
\end{theorem}

\section{Lemma on paths}

\begin{lemma}\label{lem2}
{\it\, Let $D$ be a domain in ${\Bbb R}^n,$ $n\geqslant 2,$ and let
$x_0\in \partial D.$ Assume that $E$ is  closed and nowhere dense in
$D,$ and $D$ is finitely connected on $E,$ i.e., for any $z_0\in E$
and any neighborhood $\widetilde{U}$ of $z_0$ there is a
neighborhood $\widetilde{V}\subset \widetilde{U}$ of $z_0$ such that
$(D\cap \widetilde{V})\setminus E$ consists of finite number of
components.

In addition, assume that the following condition holds: for any
neighborhood $U$ of $x_0$ there is a neighborhood $V\subset U$ of
$x_0$ such that:

\medskip
a) $V\cap D$ is connected,

\medskip
b) $(V\cap D)\setminus E$ consists at most of $m$ components,
$1\leqslant m<\infty.$

Let $x_k, y_k\in D\setminus E,$ $k=1,2,\ldots ,$ be a sequences
converging to $x_0$ as $k\rightarrow\infty.$ Then there are
subsequences $x_{k_l}$ and $y_{k_l},$ $l=1,2,\ldots ,$ belonging to
some sequence of neighborhoods $V_l,$ $l=1,2,\ldots ,$ of the point
$x_0$ such that $V_l\subset B(x_0, 2^{\,-l}),$ $l=1,2,\ldots ,$ and,
in addition, any pair $x_{k_l}$ and $y_{k_l}$ may be joined by a
path $\gamma_l$ in $V_l,$ where $\gamma_l$ contains at most $m-1$
points in~$E.$ }
\end{lemma}

\medskip
\begin{proof}
From the conditions of the lemma it follows that there exists a
sequence of neighborhoods $V_k\subset B(x_0, 2^{\,-k}),$
$k=1,2,\ldots ,$ such that $V_k\cap D$ is connected and $(V_k\cap
D)\setminus E$ consists at most of $m$ components, $1\leqslant
m<\infty.$ Now, there are subsequences $x_{k_l}$ and $y_{k_l}$ which
belong to $V_l\cap D.$ To simplify the notation, we will assume that
the sequences $x_k$ and $y_k$ themselves satisfy this condition,
i.e., $x_k, y_k\in V_k,$ $k=1,2,\ldots .$ Since $V_k\cap D$ is
connected, there is a path $\gamma_k:[0, 1]\rightarrow V_k$ such
that $\gamma_k(0)=x_k$ and $\gamma_k(1)=y_k.$

Let $K_1$ be a component of $(V_k\cap D)\setminus E$ containing
$x_k.$ If $y_k\subset K_1,$ the proof  is complete. In the contrary
case, $|\gamma_k|\cap (D\setminus K_1)\ne \varnothing.$ Let us to
show that, in this case,
\begin{equation}\label{eq4}
|\gamma_k|\cap (D\setminus \overline{K_1})\ne \varnothing\,.
\end{equation}
Indeed, since $y_k\in (V_k\cap D)\setminus E,$ there is a component
$K_*$ of $(V_k\cap D)\setminus E$ such that $y_k\in K_*.$ Observe
that $y_k\not \in \overline{K_1}.$ Indeed, in the contrary case
there is a sequence $z_s\in K_1,$ $s=1,2,\ldots ,$ such that
$z_s\rightarrow y_k$ as $s\rightarrow\infty.$ But all components of
$(V_k\cap D)\setminus E$ are closed into itself and disjoint (see,
e.g., \cite[Theorem~1.5.III]{Ku}). Thus $y_k\in K_1,$ as well. It is
possible only if $K_1=K_*,$ that contradicts the assumption
mentioned above. Therefore, the relation~(\ref{eq4}) holds, as
required.

\medskip
By~(\ref{eq4}), there is $t_1\in (0, 1)$ such that
\begin{equation}\label{eq5A}
t_1=\sup\limits_{t\geqslant 0: \gamma_k(t)\in \overline{K_1}}t\,.
\end{equation}
Set
\begin{equation}\label{eq6B}
z_1=\gamma_k(t_1)\in V_k\cap D\,.
\end{equation}
Observe that $z_1\in E.$ Indeed, in the contrary case there is a
component $K_{**}$ of $(V_k\cap D)\setminus E$ such that $z_1\in
K_{**}.$ Since $K_1$ is closed, $K_{**}=K_1.$ Since $K_1$ is open
set (see~\cite[Theorem~4.6.II]{Ku}), there is $t^{\,*}_{1}>t_1$ such
that a path $\gamma_k|_{[t_1, t^{\,*}_{1}]}$ belongs to $K_1$ yet.
But this contradicts with the definition of $t_1$ in~(\ref{eq5A}).

\medskip
Let us to show that, there is another component $K_2$ of $(V_k\cap
D)\setminus E$ such that $z_1\in\partial K_2.$ Indeed, the points
$\gamma_k(t),$ $t>t_1,$ do not belong to $\overline{K_1},$ so that
there is a sequence $z^l_1\in (V_k\cap D)\setminus K_1,$
$l=1,2\ldots,$ such that $z^l_1\rightarrow z_1$ as
$l\rightarrow\infty.$ Since $E$ is nowhere dense in $D,$ $V_k\cap D$
is open and $z_1\in V_k\cap D,$ we may consider that $z^l_1\in
(V_k\cap D)\setminus E$ for any $l\in {\Bbb N}.$ Since there are $m$
components of $(V_k\cap D)\setminus E,$ we may choose a component
$K_2$ of them which contains infinitely many elements of the
sequence $z^l_1.$ Without loss of generality, passing to a
subsequence, if need, we may consider that all elements $z^l_1\in
K_2,$ $l=1,2,\ldots .$ Thus, $z_1\in \partial K_2,$ as required.

Observe that, any component $K$ of $(V_k\cap D)\setminus E$ is
finitely connected at any point $z_0\in \partial K.$ Indeed, since
$D$ is finitely connected on $E,$ for any neighborhood
$\widetilde{U}$ of $z_0$ there is a neighborhood
$\widetilde{V}\subset \widetilde{U}$ of $z_0$ such that $(D\cap
\widetilde{V})\setminus E$ consists of finite number of components.
Now, $(K\cap \widetilde{V})\setminus E$ consists of finite number of
components, as well.

Now, due to Lemma~3.10 in \cite{Vu} we may replace $\gamma_k|_{[0,
t_1]}$ by a path belonging to $K_1$ for any $t\in [0, t_1)$ and
tending to $z_1$ as $t\rightarrow t_1-0.$ (Here $z_1$ is defined
in~(\ref{eq6B})). In order to simplify the notation, without
limiting the generality of the reasoning, we may assume that the
path $\gamma_k|_{[0, t_1]}$ already has the indicated property. Now,
there are two cases:

\medskip
a) $y_k\in K_2.$ Then we join the points $y_k$ and $z_1$ by a path
$\alpha_1:(t_1, 1]\rightarrow K_2$ belonging to $K_2$ and tending to
$z_1$ as $t\rightarrow t_1+0.$ This is possible due to finitely
connectedness of $K_2$ proved above and by Lemma~3.10 in \cite{Vu}.
Uniting paths $\gamma_k|_{[0, t_1]}$ and $\alpha_1,$ we obtain the
desired path. In particular, only one point $z_1$ belonging to this
path also belongs to $E.$

\medskip
b) $y_k\not\in K_2.$ Arguing similarly to mentioned above, we may
prove that there is $t_2\in (t_1, 1)$ such that
\begin{equation}\label{eq7D}
t_2=\sup\limits_{t\geqslant t_1: \gamma_k(t)\in \overline{K_2}}t\,.
\end{equation}
Reasoning similarly to the case with point $t_1,$ we may show that:
$$(b_1)\qquad z_2=\gamma_k(t_2)\in (V_k\cap D)\cap E\,,$$

\noindent $(b_2)$\qquad there is another component $K_3$ of
$(V_k\cap D)\setminus E,$ $K_1\ne K_3\ne K_2$ such that
$z_2\in\partial K_3,$

\noindent $(b_3)$\qquad due to Lemma~3.10 in \cite{Vu} we may
replace $\gamma_k|_{[t_1, t_2]}$ by a path belonging to $K_2$ for
any $t\in (t_1, t_2),$ tending to $z_1$ as $t\rightarrow t_1+0$ and
tending to $z_2$ as $t\rightarrow t_2-0.$ In order to simplify the
notation, without limiting the generality of the reasoning, we may
assume that the path $\gamma_k|_{[t_1, t_2]}$ already has the
indicated property.

\medskip
Now, there are two cases:

\medskip
a) $y_k\in K_3.$ Then we join the points $y_k$ and $z_3$ by a path
$\alpha_2:(t_2, 1]\rightarrow K_3$ belonging to $K_3$ and tending to
$z_2$ as $t\rightarrow t_2+0.$ This is possible due to finitely
connectedness of $K_3$ proved above and by Lemma~3.10 in \cite{Vu}.
Uniting paths $\gamma_k|_{[0, t_2]}$ and $\alpha_2,$ we obtain the
desired path. In particular, precisely two points $z_1$ and $z_2$
belonging to this path also belong to $E.$

\medskip
b) $y_k\not\in K_2.$ Arguing similarly to mentioned above, we may
prove that there is $t_3\in (t_2, 1)$ such that
\begin{equation}\label{eq7E}
t_3=\sup\limits_{t\geqslant t_2: \gamma_k(t)\in \overline{K_3}}t\,.
\end{equation}
And so on. Continuing this process, we will obtain a certain number
of points $z_1=\gamma_k(t_1), z_2=\gamma_k(t_2),
z_3=\gamma_k(t_3),\ldots ,
z_{k_{\widetilde{p}-1}}=\gamma_k(t_{\widetilde{p}-1})$ in $E$ and a
certain number of components $K_1, K_2, \ldots, K_{\widetilde{p}}$
in $(V_k\cap D)\setminus E,$ $K_1\ne K_2\ne\ldots\ne
K_{\widetilde{p}}.$ The corresponding path $\gamma_k|_{[0,
t_{\widetilde{p}-1}]}$ is a part of a path $\gamma_k$ which joins
the point $x_k$ and $z_{k_{\widetilde{p}-1}}\in K_{\widetilde{p}}$
in $V_k\cap D$ and such that
$$|\gamma_k|_{[0, t_{k_{\widetilde{p}-1}}]}|\cap E=\{z_1, z_2,
z_3,\ldots , z_{k_{\widetilde{p}-1}}\}\,.$$
Since at each step the remaining part of the path does not belong to
the union of the closures of the previous components $K_1, K_2,
\ldots, K_{\widetilde{p}-1},$ and there are only a finite number of
these components does not exceed $m,$ then $y_k\in
K_{\widetilde{p}}$ for some $1\leqslant \widetilde{p}\leqslant m.$
Then we join the points $y_k$ and $z_{\widetilde{p}-1}$ by a path
$\alpha_{\widetilde{p}-1}:(t_{\widetilde{p}-1}, 1]\rightarrow
K_{\widetilde{p}}$ belonging to $K_{\widetilde{p}}$ and tending to
$z_{\widetilde{p}-1}$ as $t\rightarrow t_{\widetilde{p}-1}+0.$ This
is possible due to finitely connectedness of $K_{\widetilde{p}}$
proved above and by Lemma~3.10 in \cite{Vu}. Uniting paths
$\gamma_k|_{[0, t_{\widetilde{p}-1}]}$ and
$\alpha_{\widetilde{p}-1},$ we obtain the desired path. In
particular, only the points $z_1, z_2,\ldots, z_{\widetilde{p}-1}$
belonging to this path also belongs to $E.$ The number of these
points does not exceed $m-1.$ Lemma is proved.~$\Box$
\end{proof}

\section{Main Lemma}

Let $D\subset {\Bbb R}^n,$ $f:D\rightarrow {\Bbb R}^n$ be a discrete
open mapping, $\beta: [a,\,b)\rightarrow {\Bbb R}^n$ be a path, and
$x\in\,f^{\,-1}(\beta(a)).$ A path $\alpha: [a,\,c)\rightarrow D$ is
called a {\it maximal $f$-lifting} of $\beta$ starting at $x,$ if
$(1)\quad \alpha(a)=x\,;$ $(2)\quad f\circ\alpha=\beta|_{[a,\,c)};$
$(3)$\quad for $c<c^{\prime}\leqslant b,$ there is no a path
$\alpha^{\prime}: [a,\,c^{\prime})\rightarrow D$ such that
$\alpha=\alpha^{\prime}|_{[a,\,c)}$ and $f\circ
\alpha^{\,\prime}=\beta|_{[a,\,c^{\prime})}.$ Here and in the
following we say that a path $\beta:[a, b)\rightarrow
\overline{{\Bbb R}^n}$ converges to the set $C\subset
\overline{{\Bbb R}^n}$ as $t\rightarrow b,$ if $h(\beta(t),
C)=\sup\limits_{x\in C}h(\beta(t), C)\rightarrow 0$ at $t\rightarrow
b.$ The following is true (see~\cite[Lemma~3.12]{MRV}).

\medskip
\begin{proposition}\label{pr3}
{\it\, Let $f:D\rightarrow {\Bbb R}^n,$ $n\geqslant 2,$ be an open
discrete mapping, let $x_0\in D,$ and let $\beta: [a,\,b)\rightarrow
{\Bbb R}^n$ be a path such that $\beta(a)=f(x_0)$ and such that
either $\lim\limits_{t\rightarrow b}\beta(t)$ exists, or
$\beta(t)\rightarrow \partial f(D)$ as $t\rightarrow b.$ Then
$\beta$ has a maximal $f$-lifting $\alpha: [a,\,c)\rightarrow D$
starting at $x_0.$ If $\alpha(t)\rightarrow x_1\in D$ as
$t\rightarrow c,$ then $c=b$ and $f(x_1)=\lim\limits_{t\rightarrow
b}\beta(t).$ Otherwise $\alpha(t)\rightarrow \partial D$ as
$t\rightarrow c.$}
\end{proposition}

\medskip
Versions of the following lemma have been repeatedly proven by
different authors, including the second co-author of this work in
the situation where the maps are homeomorphisms or open discrete
closed maps, see, for example, \cite[Lemma~2.1]{IR},
\cite[Lemma~5.16]{MRSY$_1$}, \cite[Theorem~4.6,
Theorem~13.1]{MRSY$_2$}, \cite[Theorrem~5.1]{RS}, \cite{Sev$_3$} and
\cite[Theorem~1]{Sm}. For quasiconformal mappings, results of this
kind were proved in~\cite[Theorem~17.15]{Va}
and~\cite[Section~3]{Na$_1$}. For closed quasiregular mappings it
was proved by Vuorinen and Srebro with some differences in the
formulation, see, for example, \cite[Theorem~4.10.II]{Vu} and
cf.~\cite[Theorem~4.2]{Sr}. Apparently, in the situation where the
mappings are not closed, this lemma is published for the first time,
even for quasiregular mappings.

\medskip
\begin{lemma}\label{lem1} {\it Let $p\geqslant 1,$
Let $D$ and $D^{\,\prime}$ be domains in ${\Bbb R}^n,$ $n\geqslant
2,$ $f:D\rightarrow D^{\,\prime}$ be an open discrete mapping
satisfying relations~(\ref{eq2*!A})--(\ref{eqA2}) at the point $b\in
\partial D,$ $b\ne\infty,$ with respect to $p$-modulus, $b\ne \infty,$
$f(D)=D^{\,\prime}.$

\medskip
In addition, assume that

\medskip
1) the set
$$E:=f^{\,-1}(C(f, \partial
D))$$
is nowhere dense in $D$ and $D$ is finitely connected on $E;$

\medskip
2) for any neighborhood $U$ of $x_0$ there is a neighborhood
$V\subset U$ of $x_0$ such that:

\medskip
2a) $V\cap D$ is connected,

\medskip
2b) $(V\cap D)\setminus E$ consists at most of $m$ components,
$1\leqslant m<\infty,$

\medskip
3) $D^{\,\prime}\setminus C(f, \partial D)$ consists of finite
components, each of them has a strongly accessible boundary with
respect to $p$-modulus.

Suppose that there is $\varepsilon_0=\varepsilon_0(b)>0$ and some
positive measurable function $\psi:(0, \varepsilon_0)\rightarrow
(0,\infty)$ such that
\begin{equation}\label{eq7***}
0<I(\varepsilon,
\varepsilon_0)=\int\limits_{\varepsilon}^{\varepsilon_0}\psi(t)\,dt
< \infty
\end{equation}
for sufficiently small $\varepsilon\in(0, \varepsilon_0)$ and, in
addition,
\begin{equation}\label{eq5***}
\int\limits_{A(\varepsilon, \varepsilon_0, b)}
Q(x)\cdot\psi^{\,p}(|x-b|)
 \ dm(x) =o(I^p(\varepsilon, \varepsilon_0))\,,
\end{equation}
where $A:=A(b, \varepsilon, \varepsilon_0)$ is defined
in~(\ref{eq1**}). Then $f$ has a continuous extension to $b.$}
\end{lemma}

\begin{proof}
Suppose the opposite. Due to the compactness of $\overline{{\Bbb
R}^n},$ then there are at least two sequences $x_k,$ $y_k\in D,$
$i=1,2,\ldots,$ such that $x_k,$ $y_k\in D,$ $k=1,2,\ldots,$
$x_k\rightarrow b,$ $y_k\rightarrow b$ as $k\rightarrow \infty,$ and
$f(x_k)\rightarrow y,$ $f(y_k)\rightarrow y^{\,\prime}$ as
$k\rightarrow \infty,$ while $y^{\,\prime}\ne y,$ see
Figure~\ref{fig1}.
\begin{figure}[h]
\centerline{\includegraphics[scale=0.5]{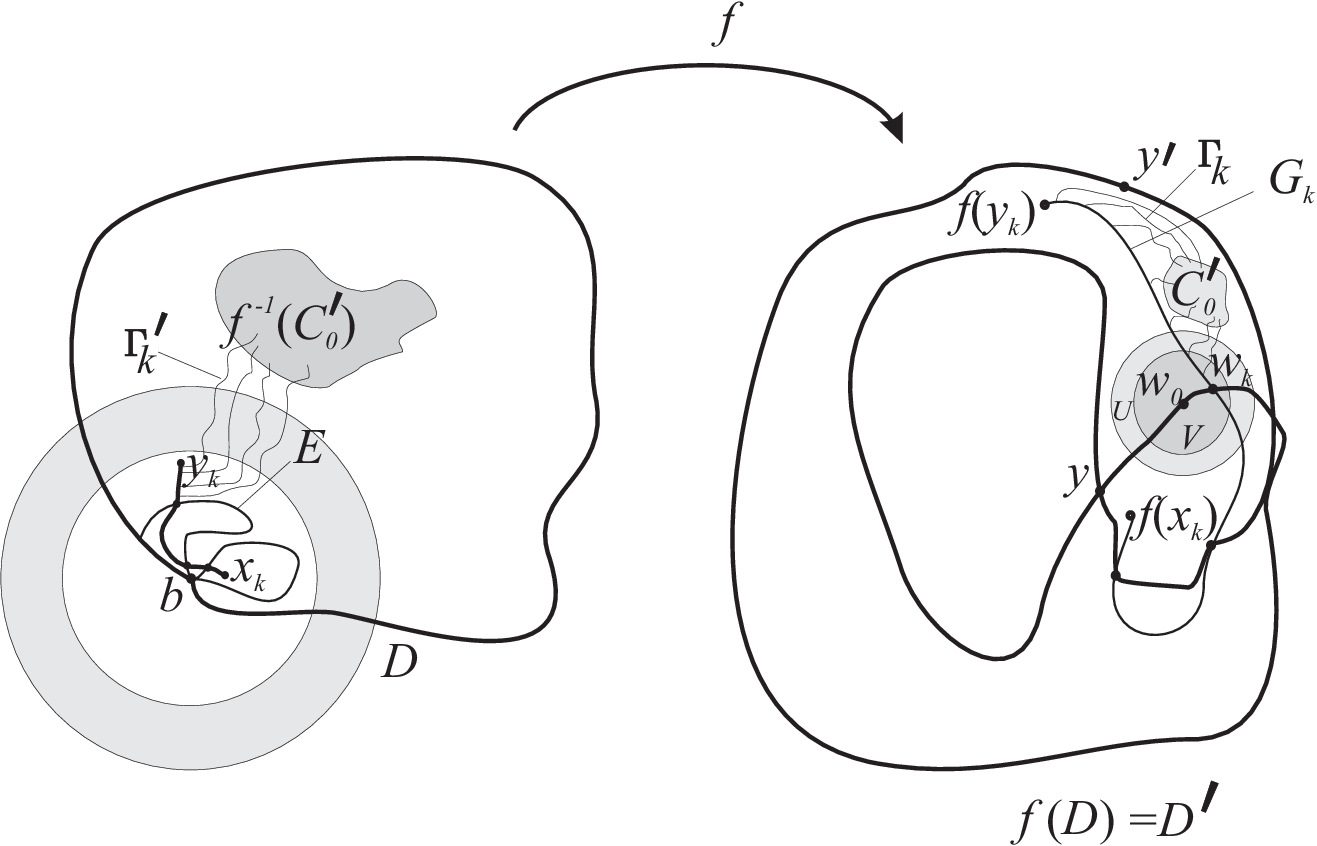}} \caption{To
the proof of Theorem~\ref{th3}}\label{fig1}
\end{figure}
In particular,
\begin{equation}\label{eq1B}
h(f(x_k), f(y_k))\geqslant \delta>0
\end{equation}
for some $\delta>0$ and all $k\in {\Bbb N},$ where $h$ is chordal
(spherical) metric.
By the assumption~2), there exists a sequence of neighborhoods
$V_k\subset B(b, 2^{\,-k}),$ $k=1,2,\ldots ,$ such that $V_k\cap D$
is connected and $(V_k\cap D)\setminus E$ consists at most of $m$
components, $1\leqslant m<\infty.$

We note that the points $x_k$ and $y_k,$ $k=1,2,\ldots, $ may be
chosen such that $x_k, y_k\not\in E.$ Indeed, since under
condition~1) the set $E$ is nowhere dense in $D,$ there exists a
sequence $x_{ki}\in D\setminus E,$ $i=1,2,\ldots ,$ such that
$x_{ki}\rightarrow x_k$ as $i\rightarrow\infty.$ Put
$\varepsilon>0.$ Due to the continuity of the mapping $f$ at the
point $x_k,$ for the number $k\in {\Bbb N}$ there is a number
$i_k\in {\Bbb N}$ such that $h(f(x_{ki_k}), f(x_k))<\frac{1}{2^k}.$
So, by the triangle inequality
$$h(f(x_{ki_k}), y)\leqslant h(f(x_{ki_k}), f(x_k))+
h(f(x_k), y)\leqslant \frac{1}{2^k}+\varepsilon\,,$$
$k\geqslant k_0=k_0(\varepsilon),$ since $f(x_k)\rightarrow y$ as
$k\rightarrow\infty$ and by the choice of $x_k$ and $y.$ Therefore,
$x_k\in D$ may be replaced by $x_{ki_k}\in D\setminus E,$ as
required. We may reason similarly for the sequence $y_k.$

\medskip
Now, by Lemma~\ref{lem1} there are subsequences $x_{k_l}$ and
$y_{k_l},$ $l=1,2,\ldots ,$ belonging to some sequence of
neighborhoods $V_l,$ $l=1,2,\ldots ,$ of the point $x_0$ such that
${\rm diam\,}V_l\rightarrow 0$ as $l\rightarrow\infty$ and, in
addition, any pair $x_{k_l}$ and $y_{k_l}$ may be joined by a path
$\gamma_l$ in $V_l,$ where $\gamma_l$ contains at most $m-1$ points
in $E.$ Without loss of generality, we may assume that the same
sequences $x_k$ and $y_k$ satisfy properties mentioned above. Let
$\gamma_k:[0, 1]\rightarrow D,$ $\gamma_k(0)=x_k$ and
$\gamma_k(1)=y_k,$ $k=1,2,\ldots .$

\medskip
Observe that, the path $f(\gamma_k)$ contains not more than $m-1$
points in $C(f, \partial D).$ In the contrary case, there are at
least $m$ such points $b_{1}=f(\gamma_k(t_1)),
b_{2}=f(\gamma_k(t_2)),\ldots, b_m=f(\gamma_k(t_1)),$ $0\leqslant
t_1\leqslant t_2\leqslant \ldots\leqslant t_m\leqslant 1.$ But now
the points $a_{1}=\gamma_k(t_1), a_{2}=\gamma_k(t_2),\ldots,
a_m=\gamma_k(t_m)$ are in $E=f^{\,-1}(C(f, \partial D))$ and
simultaneously belong to $\gamma_k.$ This contradicts the definition
of $\gamma_k.$

\medskip
Let
$$b_{1}=f(\gamma_k(t_1)), b_{2}=f(\gamma_k(t_2))\quad,\ldots,\quad
b_l=f(\gamma_k(t_1))\,,$$ $$0\leqslant t_1\leqslant t_2\leqslant
\ldots\leqslant t_l\leqslant 1,\qquad 1\leqslant l\leqslant m-1\,,$$
be points in $f(\gamma_k)\cap C(f,
\partial D).$ By the relation~(\ref{eq1B}) and due to the triangle
inequality,
\begin{equation}\label{eq6A}
\delta\leqslant h(f(x_k), f(y_k))\leqslant\sum\limits_{k=1}^{l-1}
h(f(\gamma_k(t_k)), f(\gamma_k(t_{k+1}))\,.
\end{equation}
It follows from~(\ref{eq6A}) that, there is $1\leqslant
l=l(k)\leqslant m-1$ such that such that
\begin{equation}\label{eq7}
h(f(\gamma_k(t_{l(k)})), f(\gamma_k(t_{l(k)+1}))\geqslant
\delta/l\geqslant \delta/(m-1)\,.
\end{equation}
Observe that, the set $G_k:=|\gamma_k|_{(t_{l(k)}, t_{l(k)+1})}|$
belongs to $D^{\,\prime}\setminus C(f, \partial D),$ because it does
not contain any point in $E.$ Since $D^{\,\prime}\setminus C(f,
\partial D)$ consists only of finite components, there exists at
least one a component of $D^{\,\prime}\setminus C(f, \partial D),$
containing infinitely many components of $G_k.$ Without loss of
generality, going to a subsequence, if need, we may assume that all
$G_k$ belong to one component $K$ of $D^{\,\prime}\setminus C(f,
\partial D).$

\medskip
Due to the compactness of $\overline{{\Bbb R}^n},$ we may assume
that the sequence $w_k:=f(\gamma_k(t_{l(k)})),$ $k=1,2,\ldots, $
converges to some a point $w_0\in \overline{D^{\,\prime}}.$ Let us
to show that $w_0\in C(f, \partial D).$ Indeed, there are two cases:
either $w_k=f(\gamma_k(t_{l(k)}))\in C(f, \partial D)$ for
infinitely many $k,$ or $w_k\not\in E$ for infinitely many $k\in
{\Bbb N}.$ In the first case, the inclusion $w_0\in C(f, \partial
D)$ is obvious, because $C(f, \partial D)$ is closed. In the second
case, obviously, $w_k=f(x_k),$ but this sequence converges to $y\in
C(f, \partial D)$ by the assumption.

By the assumption, each component of the set $D^{\,\prime}\setminus
C(f, \partial D)$ has a strongly accessible boundary with respect to
$p$-modulus. In this case, for any neighborhood $U$ of the point
$w_0\in
\partial K$ there is a compact set $C_0^{\,\prime}\subset
D^{\,\prime},$ a neighborhood $V$ of a point $y,$ $V\subset U,$ and
a number $P>0$ such that
\begin{equation}\label{eq1}
M_p(\Gamma(C_0^{\,\prime}, F, K))\geqslant P
>0
\end{equation}
for any continua $F,$ intersecting $\partial U$ and $\partial V.$
Choose a neighborhood $U$ of $w_0$ with $h(U)< \delta/2(m-1),$ where
$\delta$ is from~(\ref{eq1B}). Let $C_0^{\,\prime}$ and $V$ be a
compact set and a neighborhood corresponding to $w_0.$

Observe that, $G_k$ contains some a continuum $\widetilde{G}_k$ in
$K$ intersecting $\partial U$ and $\partial V$ for sufficiently
large $k.$ Indeed, by the construction of $G_k,$ there is a sequence
of points $a_{s, k}:=f(\gamma_k(p_s))\rightarrow
w_k:=f(\gamma_k(t_{l(k)}))$ as $s\rightarrow \infty$ and $b_{s,
k}=f(\gamma_k(q_s))\rightarrow f(\gamma_k(t_{l(k)+1}))$ as
$s\rightarrow \infty,$ where $t_{l(k)})<
p_s<q_s<\gamma_k(t_{l(k)+1}))$ and $a_{s, k}, b_{s, k}\in K.$
By the triangle inequality and by~(\ref{eq7})
$$
h(a_{s, k}, b_{s, k})\geqslant h(b_{s, k}, w_k)-h(w_k, a_{s,
k})\geqslant
$$
\begin{equation}\label{eq11}
h(f(\gamma_k(t_{l(k)+1})), w_k)-h(f(\gamma_k(t_{l(k)+1})), b_{s,
k})-h(w_k, a_{s, k})\geqslant
\end{equation}
$$\delta/(m-1)-h(f(\gamma_k(t_{l(k)+1}), b_{s,
k}))-h(w_k, a_{s, k})\,.$$
Since $a_{s, k}:=f(\gamma_k(p_s))\rightarrow
w_k:=f(\gamma_k(t_{l(k)}))$ as $s\rightarrow \infty$ and $b_{s,
k}=f(\gamma_k(q_s))\rightarrow f(\gamma_k(t_{l(k)+1}))$ as
$s\rightarrow \infty,$ it follows from the last inequality that
there is $s=s(k)\in {\Bbb N}$ such that
\begin{equation}\label{eq12}
h(a_{s(k), k}, b_{s(k), k})\geqslant\delta/2(m-1)\,.
\end{equation}
Since $V$ is open, there is some neighborhood $U_k$ of $w_k$ such
$U_k\subset V.$ Since $a_{s, k}\rightarrow
w_k:=f(\gamma_k(t_{l(k)}))$ as $s\rightarrow \infty,$ we may assume
that $a_{s(k), k}\in V.$
Now, we set
$$\widetilde{G}_k:=f(\gamma_k)|_{[p_{s(k)}, q_{s(k)}]}\,.$$
In other words, $\widetilde{G}_k$ is a part of the path
$f(\gamma_k)$ between points $a_{s(k), k}$ and $b_{s(k), k}.$ Let us
to show that $\widetilde{G}_k$ intersects $\partial U$ and $\partial
V.$ Indeed, by the mentioned above, $a_{s(k), k}\subset V,$ so that
$V\cap \widetilde{G}_k\ne\varnothing.$ In particular, $U\cap
\widetilde{G}_k\ne\varnothing,$ because $V\subset U.$ On the other
hand, by (\ref{eq12}) $h(\widetilde{G}_k)\geqslant\delta/2(m-1),$
however, $h(U)< \delta/2(m-1)$ by the choice of $U.$ In particular,
$h(V)<\delta/2(m-1).$ It follows from this that
$$(\overline{{\Bbb R}^n}\setminus U)\cap
\widetilde{G}_k\ne\varnothing\,, \qquad(\overline{{\Bbb
R}^n}\setminus V)\cap \widetilde{G}_k\ne\varnothing\,.$$
Now, by \cite[Theorem~1.I.5, \S46]{Ku}
$$\partial U\cap
\widetilde{G}_k\ne\varnothing\,, \partial V\cap
\widetilde{G}_k\ne\varnothing\,,$$
as required.
Now, by~(\ref{eq1}),
\begin{equation}\label{eq1C}
M_p(\Gamma(\widetilde{G}_k, F, K))\geqslant P
>0\,, \qquad k=1,2,\ldots\,.
\end{equation}
Let us to show that, the relation~(\ref{eq1C}) contradicts with the
definition of $f$ in~(\ref{eq2*!A}) together with the
conditions~(\ref{eq7***})--(\ref{eq5***}). Indeed, let us denote by
$\Gamma_k$ the family of all half-open paths $\beta_k:[a,
b)\rightarrow \overline{{\Bbb R}^n}$ such that $\beta_k(a)\in
|f(\gamma_k)|,$ $\beta_k(t)\in K$ for all $t\in [a, b)$ and,
moreover, $\lim\limits_{t\rightarrow b-0}\beta_k(t):=B_k\in
C_0^{\,\prime}.$ Obviously, by~(\ref{eq1C})
\begin{equation}\label{eq4A}
M_p(\Gamma_k)=M_p(\Gamma(\widetilde{G}_k, F, K))\geqslant P
>0\,, \qquad k=1,2,\ldots\,.
\end{equation}
Consider the family $\Gamma_k^{\,\prime}$ of maximal $f$-liftings
$\alpha_k:[a, c)\rightarrow D$ of the family $\Gamma_k$ starting at
$|\gamma_k|;$ such a family exists by Proposition~\ref{pr3}.

Observe that, the situation when $\alpha_k\rightarrow \partial D$ as
$k\rightarrow\infty$ is impossible. Suppose the opposite: let
$\alpha_k(t)\rightarrow \partial D$ as $t\rightarrow c.$ Let us
choose an arbitrary sequence $\varphi_m\in [0, c)$ such that
$\varphi_m\rightarrow c-0$ as $m\rightarrow\infty.$ Since the space
$\overline{{\Bbb R}^n}$ is compact, the boundary $\partial D$ is
also compact as a closed subset of the compact space. Then there
exists $w_m\in
\partial D$ such that
\begin{equation}\label{eq7B}
h(\alpha_k(\varphi_m), \partial D)=h(\alpha_k(\varphi_m), w_m)
\rightarrow 0\,,\qquad m\rightarrow \infty\,.
\end{equation}
Due to the compactness of $\partial D$, we may assume that
$w_m\rightarrow w_0\in \partial D$ as $m\rightarrow\infty.$
Therefore, by the relation~(\ref{eq7B}) and by the triangle
inequality
\begin{equation}\label{eq8B}
h(\alpha_k(\varphi_m), w_0)\leqslant h(\alpha_k(\varphi_m),
w_m)+h(w_m, w_0)\rightarrow 0\,,\qquad m\rightarrow \infty\,.
\end{equation}
On the other hand,
\begin{equation}\label{eq9B}
f(\alpha_k(\varphi_m))=\beta_k(\varphi_m)\rightarrow \beta(c)
\,,\quad m\rightarrow\infty\,,
\end{equation}
because by the construction the path $\beta_k(t),$ $t\in [a, b],$
lies in $K\subset D^{\,\prime}\setminus C(f, \partial D)$ together
with its finite ones points. At the same time, by~(\ref{eq8B})
and~(\ref{eq9B}) we have that $\beta_k(c)\in C(f,
\partial D).$ The inclusions $\beta_k\subset D^{\,\prime}\setminus C(f, \partial D)$
and $\beta_k(c)\in C(f,
\partial D)$ contradict each other.

\medskip
Therefore, by Proposition~\ref{pr3} $\alpha_k\rightarrow x_1\in D$
as $t\rightarrow c-0,$ and $c_b$ and $f(\alpha_k(b))=f(x_1).$ In
other words, the $f$-lifting $\alpha_k$ is complete, i.e.,
$\alpha_k:[a, b]\rightarrow D.$ Besides that, it follows from that
$\alpha_k(b)\in f^{\,-1}(C^{\,\prime}_0).$

\medskip
Observe that, there is $r^{\,\prime}_0>0$ such that
\begin{equation}\label{eq13}
h(f^{\,-1}(C^{\,\prime}_0), \partial D)\geqslant r^{\,\prime}_0>0\,.
\end{equation}
In the contrary case, there is $z_k\in f^{\,-1}(C^{\,\prime}_0)$
such that $z_k\rightarrow z_0\in \partial D$ as
$k\rightarrow\infty.$ Now, due to the compactness of
$\overline{{\Bbb R}^n}$ the sequence $f(z_k)$ converges to some
$\omega_0$ as $k\rightarrow\infty.$ But, in this case, $\omega_0$ in
$C(f, \partial D)$ and simultaneously $\omega_0\in K\subset
D^{\,\prime}\setminus C(f, \partial D)$ because $f(z_k)\in
C^{\,\prime}_0\subset K$ and $C^{\,\prime}_0$ is closed. We have
obtained a contradiction, so that the relation~(\ref{eq13}) holds.

\medskip
Since $b\ne\infty,$ it follows from~(\ref{eq13}) that
\begin{equation}\label{eq14}
f^{\,-1}(C^{\,\prime}_0)\subset D\setminus B(b, r_0)\,.
\end{equation}
Let $k$ be such that $2^{\,-k}<\varepsilon_0.$ We may consider that
$\varepsilon_0<r_0.$ Due to~(\ref{eq14}), we may show that
\begin{equation}\label{eq15}
\Gamma_k^{\,\prime}>\Gamma(S(b, 2^{\,-k}), S(b, \varepsilon_0), D)
\end{equation}
(see \cite[Theorem~1.I.5, \S46]{Ku}). Observe that the function
$$\eta(t)=\left\{
\begin{array}{rr}
\psi(t)/I(2^{\,-k}, \varepsilon_0), &   t\in (2^{\,-k},
\varepsilon_0),\\
0,  &  t\in {\Bbb R}\setminus (2^{\,-k}, \varepsilon_0)\,,
\end{array}
\right. $$
where $I(2^{\,-k}, \varepsilon_0)$ is defined in~(\ref{eq7***}),
satisfies~(\ref{eqA2}) for $r_1:=2^{\,-k},$ $r_2:=\varepsilon_0.$
Therefore, by the definition of the mapping in~(\ref{eq2*!A}) and
by~(\ref{eq15}) we obtain that
\begin{equation}\label{eq11*}
M_p(f(\Gamma^{\,\prime}_k))\leqslant \Delta(k)\,,
\end{equation}
where $\Delta(k)\rightarrow 0$ as $k\rightarrow \infty.$ However,
$\Gamma_k=f(\Gamma_k^{\,\prime}).$ Thus, by~(\ref{eq11*}) we obtain
that
\begin{equation}\label{eq3}
M_p(\Gamma_k)= M_p(f(\Gamma_k^{\,\prime}))\leqslant
\Delta(k)\rightarrow 0
\end{equation}
as $k\rightarrow \infty.$ However, the relation~(\ref{eq3}) together
with the inequality~(\ref{eq4A}) contradict each other, which proves
the lemma. $\Box$
\end{proof}

\medskip
{\it Proof of Theorem~\ref{th3}} immediately follows by
Lemma~\ref{lem1} and Lemma~1.3 in~\cite{Sev$_4$}, excluding the
equality $f(\overline{D})=\overline{D^{\,\prime}}.$ The proof of the
latter repeats the last part of the proof of Theorem~3.1 in
\cite{SSD}.~$\Box$

\section{Some remarks on uniform domains}

A domain $D\subset {\Bbb R}^n,$ $n\geqslant 2,$ is called a {\it
uniform} domain with respect to $p$-modulus, $p\geqslant 1,$ if, for
each $r>0,$ there is $\delta>0$ such that $M_{p}(\Gamma(F, F^{\,*},
D))\geqslant\delta$ whenever $F$ and $F^{\,*}$ are continua of $D$
with $h(F)\geqslant r$ and $h(F^{\,*})\geqslant r.$ Domains $D_i,$
$i\in I,$ are said to be {\it equi-uniform} domains with respect to
$p$-modulus, if, for $r>0,$ the modulus condition above is satisfied
by each $D_i$ with the same number $\delta.$ It should be noted that
the proposed concept of a uniform domain has, generally speaking, no
relation to definition, introduced in~\cite{MSa}. Note that, uniform
domains with respect to $p$-modulus have strongly accessible
boundaries with respect to $p$-modulus, see
\cite[Remark~1]{SevSkv$_1$}.

\medskip
\begin{remark}\label{rem2}
The statement of Lemma~\ref{lem1} remains true if condition 3) in
this lemma is replaced by the conditions:

($3_\textbf{a}$) the family of components of $D^{\,\prime}\setminus
C(f, \partial D)$ is {\it equi-uniform} with respect to $p$-modulus
and, besides that,

(\textbf{$3_\textbf{b}$}) there is $\delta_*>0$ such that, in any
component $K$ of $D^{\,\prime}\setminus  C(f, \partial D)$ there is
a compactum $F\subset K$ such that $h(F)\geqslant \delta_*$ and
$h(f^{\,-1}(F),
\partial D)\geqslant \delta_*>0.$

\medskip
Indeed, just as in the proof of Lemma~\ref{lem1}, we will prove this
statement by contradiction. The proof up to relation (3.5) is
repeated verbatim. Further, just as in the proof of relations
(\ref{eq11})-(\ref{eq12}), it may be shown that there is a
subcontinuum~$\widetilde{G}_k$ in $G_k$ such that
$h(\widetilde{G}_k)\geqslant\delta/2(m-1).$

Let $K_k$ be a component of $D^{\,\prime}\setminus  C(f, \partial
D)$ containing $\widetilde{G}_k,$ and let $F_k:=F$ be compactum in
$K_k$ from the condition~$3_\textbf{b}.$ Due to the equi-uniformity
of $K_k$ (see the condition~$3_\textbf{b}$), there exists $P>0$ such
that
\begin{equation}\label{eq4B}
M_p(\Gamma_k)=M_p(\Gamma(\widetilde{G}_k, F_k, K_k))\geqslant P
>0\,, \qquad k=1,2,\ldots\,.
\end{equation}
Here $\Gamma_k$ denotes the family of all half-open paths
$\beta_k:[a, b)\rightarrow \overline{{\Bbb R}^n}$ such that
$\beta_k(a)\in |f(\gamma_k)|,$ $\beta_k(t)\in K$ for all $t\in [a,
b)$ and, moreover, $\lim\limits_{t\rightarrow b-0}\beta_k(t):=B_k\in
C_0^{\,\prime}.$

Consider the family $\Gamma_k^{\,\prime}$ of maximal $f$-liftings
$\alpha_k:[a, c)\rightarrow D$ of the family $\Gamma_k$ starting at
$|\gamma_k|;$ such a family exists by Proposition~\ref{pr3}.
Similarly to the proof of Lemma~\ref{lem1} we may prove that all
these liftings are complete, i.e., any path $\alpha_k:[a,
c)\rightarrow D$ is such that $c=b,$ $\alpha_k(b)\in f^{\,-1}(F_k).$
Due to the condition~$3_{\textbf{b}}$,
\begin{equation}\label{eq14A}
f^{\,-1}(K_k)\subset D\setminus B(b, r_0)\,.
\end{equation}
Arguing further in a similar way to the rest of the proof of
Lemma~\ref{lem1}, we arrive at the relation
\begin{equation}\label{eq3B}
M_p(\Gamma_k)= M_p(f(\Gamma_k^{\,\prime}))\leqslant
\Delta(k)\rightarrow 0
\end{equation}
as $k\rightarrow \infty.$ The relations~(\ref{eq4B})
and~(\ref{eq3B}) contradict each other.
\end{remark}

\section{Some examples}

\begin{example}\label{ex4}
Consider the disk $D=B(1, 1)=\{z\in {\Bbb C}: |z-1|<1\}$ on the
complex plane. Let us define a mapping $f:D\rightarrow {\Bbb C}$ as
follows: $f(z)=z^4.$ The Figure~\ref{fig2} schematically shows the
image of the domain $D$ under the mapping $f;$ the complete image is
marked in gray, the domain that is mapped twice in the disk $B(1,
1)$ is marked in a darker color.
\begin{figure}[h]
\centerline{\includegraphics[scale=0.35]{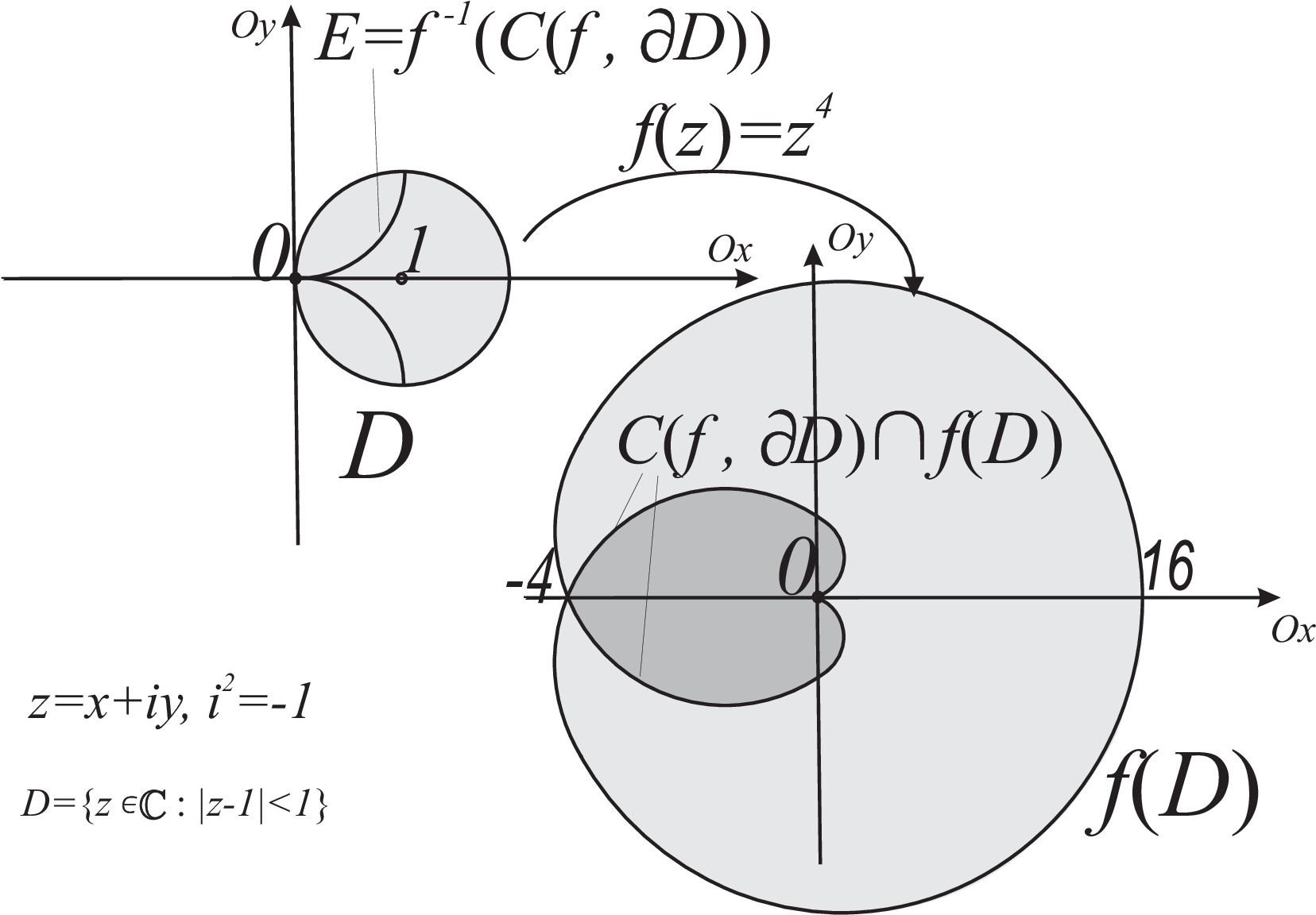}} \caption{An
open quasiregular mapping that satisfies conditions of
Theorem~\ref{th3} }\label{fig2}
\end{figure}
The image of the boundary of the domain $D$ has one point of
self-intersection $w=(-4, 0).$ Paths
$z_1(\varphi)=16\cos^4\varphi\cdot e^{i\cdot4\varphi},$
$\frac{\pi}{4}<\varphi \leqslant \frac{\pi}{2},$ and
$z_2(\varphi)=16\cos^4\varphi\cdot e^{i\cdot 4\varphi},$
$-\frac{\pi}{2}\leqslant\varphi<-\frac{\pi}{4},$ is a part of the
image of the boundary of $D$ that lies inside the mapped domain. It
is not difficult to see that, the complete pre-image in $D$ of this
paths consists of two paths
$\widetilde{z_1}(\varphi)=2\cos\varphi\cdot
e^{i\left(\varphi+\frac{\pi}{2}\right)},$ $\frac{\pi}{4}<\varphi
\leqslant \frac{\pi}{2},$ and
$\widetilde{z_2}(\varphi)=2\cos\varphi\cdot
e^{i\left(\varphi-\frac{\pi}{2}\right)},$
$-\frac{\pi}{2}\leqslant\varphi<-\frac{\pi}{4},$ and there are no
others.

Note that, the mapping $f$ satisfies all conditions of
Theorem~\ref{th3}. Indeed, the set $E:=f^{\,-1}(C(f, \partial D))$
is nowhere dense in $D,$ because $E$ consists only on paths
$\widetilde{z_1}(\varphi)$ and $\widetilde{z_2}(\varphi)$ mentioned
above. Also, $D$ is finitely connected on $E,$ because $D$ is
tho-connected at any inner point of $D$ which belong to the paths
$\widetilde{z_1}(\varphi)$ and $\widetilde{z_2}(\varphi).$
Obviously, $D$ is locally connected at any point of $\partial D,$
while it is three-connected with respect to $E$ at the origin,
two-connected with respect to $E$ at the points $1\pm i$ and locally
connected with respect to $E$ at the rest points. In addition,
$f(D)\setminus C(f, \partial D)$ consists of two components.
Finally, $f(z)=z^4$ is a quasiregular mapping, therefore satisfies
the relation~(\ref{eq2*!A}) for the function $Q(y)\equiv 1$ (see,
e.g., \cite[Theorem~8.1.II]{Ri}. Observe that, $Q(x)\equiv 1$
satisfies simultaneously the conditions $Q\in FMO(x_0)$ at any
$x_0\in {\Bbb R}^n$ and~(\ref{eq6}), as well.
\end{example}

\medskip
\begin{example}\label{ex5}
It is not difficult to construct a similar mapping with some
unbounded function $Q$ in~(\ref{eq2*!A}). For this purpose, let $D$
and $f$ be a domain and a mapping from the previous example,
correspondingly. Let us consider an arbitrary point $z_0\in
D\setminus E,$ where $E:=f^{\,-1}(C(f, \partial D)),$ and let
$0<r_0<\min\{{\rm dist\,}(z_0, \partial D), {\rm dist\,}(z_0, E)\}.$
Put
$$h(z)=\frac{r_0(z-z_0)}{|z-z_0|\log\frac{e}{|(z-z_0)|}}+z_0,\qquad
h(z_0)=z_0\,.$$
Reasoning similarly to~\cite[Proposition~6.3]{MRSY$_2$}, we may show
that $h$ satisfies the relations~(\ref{eq2*!A})--(\ref{eqA2}) in
each point $b\in
\partial D$ for $Q=Q(x)=\log\left(\frac{e}{|x|}\right)$ with $p=2.$

Put
$$F(z)=\begin{cases}(f\circ h)(z)\,,& z\in B(z_0, r_0)\,, \\
(f(z)-z_0)/r_0\,, & z\in D\setminus B(z_0, r_0)
\end{cases}\,.$$
Since $f$ satisfies the relations~(\ref{eq2*!A})--(\ref{eqA2}) with
$Q\equiv 1,$ the mapping $F$ satisfies the
relations~(\ref{eq2*!A})--(\ref{eqA2}) in each point $b\in
\partial D$ for $Q=Q(x)=\log\left(\frac{e}{|x|}\right)$ with $p=2.$
Note that, $Q$ satisfies the condition~(\ref{eq6}) for any point
$b\in \partial D,$ because $Q$ is locally bounded in $\partial D.$
Observe that, the mapping $F:D\rightarrow {\Bbb C}$ satisfies all
conditions of Theorem~\ref{th3}. In particular, $F$ is constructed
so that it does not change the geometry of $D.$ Besides that, $F$ is
open and discrete as a superposition of a homeomorphism $g$ with
some quasiregular mapping. Now, by Theorem~\ref{th3} $F$ has a
continuous extension to any point $b\in\partial D.$ All conditions
regarding the geometry of a domain are also hold, see
Example~\ref{ex4}.
\end{example}

\section{On the equicontinuity of families in the closure of a domain}

For mappings with branching satisfying~(\ref{eq2*!A})--(\ref{eqA2}),
we have obtained several different theorems on the equicontinuity of
families of mappings in the closure of a domain. In particular,
in~\cite{SevSkv$_1$} we are talking about families of mappings with
a certain condition on the pre-image of a continuum, the diameter of
which does not decrease, and in~\cite{SevSkv$_2$} we assumed the
presence of one normalization condition. In these works, the
mappings were assumed to be closed. The main goal of this section is
to extend these results to mappings that are not closed.

\medskip
Given $p\geqslant 1,$ $\delta>0,$ closed sets $E, E_*, F$ in
$\overline{{\Bbb R}^n},$ $n\geqslant 2,$ a domain $D\subset {\Bbb
R}^n$ and a Lebesgue measurable function $Q:D\rightarrow [0,
\infty]$ let us denote by $\frak{R}^{E_*, E, F}_{Q, \delta, p}(D)$
the family of all open discrete mappings $f:D\rightarrow
\overline{{\Bbb R}^n}\setminus F$ satisfying the
conditions~(\ref{eq2*!A})--(\ref{eqA2}) at the point $x_0\in
\overline{D}$ such that:

\medskip
1) $C(f, \partial D)\subset E_*,$

\medskip
2) in any component $K$ of $D^{\,\prime}_f\setminus  E_*,$
$D^{\,\prime}_f:=f(D)$ there is a continuum $K_f\subset K$ such that
$h(K_f)\geqslant \delta$ and $h(f^{\,-1}(K_f), \partial D)\geqslant
\delta>0,$

\medskip
3) $f^{\,-1}(E_*)\subset E.$

\medskip
\begin{theorem}\label{th4} {\it\, Let $p\geqslant 1,$
let $D$ be a domain in ${\Bbb R}^n,$ $n\geqslant 2.$ Assume that:

\medskip
1) the set $E$ is nowhere dense in $D,$ and $D$ is finitely
connected on $E,$ i.e., for any $z_0\in E$ and any neighborhood
$\widetilde{U}$ of $z_0$ there is a neighborhood
$\widetilde{V}\subset \widetilde{U}$ of $z_0$ such that $(D\cap
\widetilde{V})\setminus E$ consists of finite number of components;

\medskip
2) for any $x_0\in\partial D$ there is $m=m(x_0)\in {\Bbb N},$
$1\leqslant m<\infty$ such that the following is true: for any
neighborhood $U$ of $x_0$ there is a neighborhood $V\subset U$ of
$x_0$ and  such that:

\medskip
2a) $V\cap D$ is connected,

\medskip
2b) $(V\cap D)\setminus E$ consists at most of $m$ components.

\medskip
Let for $p=n$ the set $F$ have positive capacity, and for $n-1<p<n$
it is an arbitrary closed set.

\medskip
Suppose that, for any $x_0\in\partial D$ at least one of the
following conditions is satisfied: 1) a function $Q$ has a finite
mean oscillation at $x_0;$ 2)
$q_{x_0}(r)\,=\,O\left(\left[\log{\frac1r}\right]^{n-1}\right)$ as
$r\rightarrow 0;$ 3) the condition
\begin{equation}\label{eq6D}
\int\limits_{0}^{\delta(x_0)}\frac{dt}{t^{\frac{n-1}{p-1}}
q_{x_0}^{\,\prime\,\frac{1}{p-1}}(t)}=\infty
\end{equation}
holds for some $\delta(x_0)>0.$

Let the family of all components of $D^{\,\prime}_f\setminus E_*$ is
equi-uniform over $f\in\frak{R}^{E_*, E, F}_{Q, \delta, p}(D)$ with
respect to $p$-modulus. Then every $f\in\frak{R}^{E_*, E, F}_{Q,
\delta, p}(D)$ has a continuous extension to $\partial D$ and a
family $\frak{R}^{E_*, E, F}_{Q, \delta, p}(\overline{D}),$
consisting of all extended mappings $\overline{f}:
\overline{D}\rightarrow \overline{{\Bbb R}^n},$ is equicontinuous at
$\overline{D}.$ }
\end{theorem}

\medskip
The proof of Theorem~\ref{th3} is based on the following lemma.

\medskip
\begin{lemma}\label{lem3} {\it\, Let $p\geqslant 1,$
let $D$ be a domain in ${\Bbb R}^n,$ $n\geqslant 2.$ Assume that:

\medskip
1) the set $E$ is nowhere dense in $D,$ and $D$ is finitely
connected on $E,$ i.e., for any $z_0\in E$ and any neighborhood
$\widetilde{U}$ of $z_0$ there is a neighborhood
$\widetilde{V}\subset \widetilde{U}$ of $z_0$ such that $(D\cap
\widetilde{V})\setminus E$ consists of finite number of components;

\medskip
2) for any $x_0\in\partial D$ there is $m=m(x_0)\in {\Bbb N},$
$1\leqslant m<\infty$ such that the following is true: for any
neighborhood $U$ of $x_0$ there is a neighborhood $V\subset U$ of
$x_0$ and  such that:

\medskip
2a) $V\cap D$ is connected,

\medskip
2b) $(V\cap D)\setminus E$ consists at most of $m$ components.

\medskip
Let for $p=n$ the set $F$ have positive capacity, and for $n-1<p<n$
it is an arbitrary closed set.

\medskip
Suppose that, for any $x_0\in \partial D$ there is
$\varepsilon_0=\varepsilon_0(x_0)>0$ and some positive measurable
function $\psi:(0, \varepsilon_0)\rightarrow (0,\infty)$ such that
\begin{equation}\label{eq8}
0<I(\varepsilon,
\varepsilon_0)=\int\limits_{\varepsilon}^{\varepsilon_0}\psi(t)\,dt
< \infty
\end{equation}
for sufficiently small $\varepsilon\in(0, \varepsilon_0)$ and, in
addition,
\begin{equation}\label{eq9}
\int\limits_{A(\varepsilon, \varepsilon_0, x_0)}
Q(x)\cdot\psi^{\,p}(|x-x_0|)
 \ dm(x) =o(I^p(\varepsilon, \varepsilon_0))\,,
\end{equation}
where $A:=A(x_0, \varepsilon, \varepsilon_0)$ is defined
in~(\ref{eq1**}).

Let the family of all components of $D^{\,\prime}_f\setminus E_*$ is
equi-uniform over $f\in\frak{R}^{E_*, E, F}_{Q, \delta, p}(D)$ with
respect to $p$-modulus. Then every $f\in\frak{R}^{E_*, E, F}_{Q,
\delta, p}(D)$ has a continuous extension to $\partial D$ and a
family $\frak{R}^{E_*, E, F}_{Q, \delta, p}(\overline{D}),$
consisting of all extended mappings $\overline{f}:
\overline{D}\rightarrow \overline{{\Bbb R}^n},$ is equicontinuous at
$\overline{D}.$}
\end{lemma}

\medskip
\begin{proof}
The equicontinuity of $\frak{R}^{E_*, E, F}_{Q, \delta, p}(D)$ in
$D$ follows by~\cite[Lemma~4.2]{SalSev} for the case of $p=n$ and
\cite[Lemma~2.4]{GSS} for $n-1<p< n.$

Since any component of $D^{\,\prime}_f\setminus E_*$ is uniform with
respect to $p$-modulus, it has strongly accessible boundary with
respect to $p$-modulus (see \cite[Remark~1]{SevSkv$_1$}). Now the
possibility of continuous extension of any $f\in\frak{R}^{E_*, E,
F}_{Q, \delta, p}(D)$ to $\partial D$ follows from Lemma~\ref{lem1}
and Remark~\ref{rem2}. It remains to prove the equicontinuity of the
extended family $\frak{R}^{E_*, E, F}_{Q, \delta, p}(\overline{D})$
in~$\partial D.$

Suppose the opposite. Then there is $x_0\in\partial D,$
$\varepsilon_0>0,$ a sequence $x_k\rightarrow x_0,$
$x_k\in\overline{D},$ and $f_k\in \frak{R}^{E_*, E, F}_{Q, \delta,
p}(\overline{D})$ such that
\begin{equation}\label{eq12A}
h(f_k(x_k), f_k(x_0))\geqslant \varepsilon_0\,.
\end{equation}
Since $f_k$ has a continuous extension to $\partial D,$ we may
consider that $x_k\in D.$ In addition, it follows from~(\ref{eq12A})
that we may find a sequence $x^{\,\prime}_k\in D,$ $k=1,2,\ldots ,$
such that $x^{\,\prime}_k\rightarrow x_0$ and
\begin{equation}\label{eq13A}
h(f_k(x_k), f_k(x^{\,\prime}_k))\geqslant \varepsilon_0/2\,.
\end{equation}
By the assumption~2), there exists a sequence of neighborhoods
$V_k\subset B(x_0, 2^{\,-k}),$ $k=1,2,\ldots ,$ such that $V_k\cap
D$ is connected and $(V_k\cap D)\setminus E$ consists of $m$
components, $1\leqslant m<\infty.$

We note that the points $x_k$ and $x^{\,\prime}_k,$ $k=1,2,\ldots, $
may be chosen such that $x_k, x^{\,\prime}_k\not\in E.$ Indeed,
since under condition~1) the set $E$ is nowhere dense in $D,$ there
exists a sequence $x_{ki}\in D\setminus E,$ $i=1,2,\ldots ,$ such
that $x_{ki}\rightarrow x_k$ as $i\rightarrow\infty.$ Put
$\varepsilon>0.$ Due to the continuity of the mapping $f_k$ at the
points $x_k$ and $x^{\,\prime}_k,$ for the number $k\in {\Bbb N}$
there are numbers $i_k, j_k\in {\Bbb N}$ such that $h(f_k(x_{ki_k}),
f_k(x_k))<\frac{1}{2^k}$ and $h(f_k(x^{\,\prime}_{kj_k}),
f_k(x^{\,\prime}_k))<\frac{1}{2^k}.$ Now, by~(\ref{eq13A}) and by
the triangle inequality,
$$
h(f_k(x_{ki_k}), f_k(x^{\,\prime}_{kj_k}))\geqslant
h(f_k(x^{\,\prime}_k), f_k(x_{ki_k}))-h(f_k(x^{\,\prime}_k),
f_k(x^{\,\prime}_{kj_k}))\geqslant
$$
\begin{equation}\label{eq14B}\geqslant h(f_k(x_k), f_k(x^{\,\prime}_k))-h(f_k(x_k),
f_k(x_{ki_k}))-h(f_k(x_k), f_k(x^{\,\prime}_{kj_k}))\geqslant
\end{equation}
$$\geqslant\varepsilon_0/2-\frac{2}{2^k}\geqslant \varepsilon_0/4$$
for sufficiently large $k.$ Due to~(\ref{eq14B}), we may assume that
$x_k, x^{\,\prime}_k\not\in E,$  as required.

\medskip
Now, by Lemma~\ref{lem1} there are subsequences $x_{k_l}$ and
$x^{\,\prime}_{k_l},$ $l=1,2,\ldots ,$ belonging to some sequence of
neighborhoods $V_l,$ $l=1,2,\ldots ,$ of the point $x_0$ such that
${\rm diam\,}V_l\rightarrow 0$ as $l\rightarrow\infty$ and, in
addition, any pair $x_{k_l}$ and $x^{\,\prime}_{k_l}$ may be joined
by a path $\gamma_l$ in $V_l,$ where $\gamma_l$ contains at most
$m-1$ points in $E.$ Without loss of generality, we may assume that
the same sequences $x_k$ and $y_k$ satisfy properties mentioned
above. Let $\gamma_k:[0, 1]\rightarrow D,$ $\gamma_k(0)=x_k$ and
$\gamma_k(1)=y_k,$ $k=1,2,\ldots .$

\medskip
Observe that, the path $f_k(\gamma_k)$ contains not more than $m-1$
points in $E_*.$ In the contrary case, there are at least $m$ such
points $b_{1}=f_k(\gamma_k(t_1)), b_{2}=f_k(\gamma_k(t_2)),\ldots,
b_m=f_k(\gamma_k(t_1)),$ $0\leqslant t_1\leqslant t_2\leqslant
\ldots\leqslant t_m\leqslant 1.$ But now the points
$a_{1}=\gamma_k(t_1), a_{2}=\gamma_k(t_2),\ldots, a_m=\gamma_k(t_m)$
are in $E=f^{\,-1}(E_*)$ and simultaneously belong to $\gamma_k.$
This contradicts the definition of $\gamma_k.$

\medskip
Let
$$b_{1}=f_k(\gamma_k(t_1)), b_{2}=f_k(\gamma_k(t_2))\quad,\ldots,\quad
b_l=f_k(\gamma_k(t_1))\,,$$ $$0\leqslant t_1\leqslant t_2\leqslant
\ldots\leqslant t_l\leqslant 1,\qquad 1\leqslant l\leqslant m-1\,,$$
be points in $f_k(\gamma_k)\cap E_*.$ By the relation~(\ref{eq13A})
and due to the triangle inequality,
\begin{equation}\label{eq6C}
\varepsilon_0/2\leqslant h(f_k(x_k),
f_k(y_k))\leqslant\sum\limits_{k=1}^{l-1} h(f_k(\gamma_k(t_k)),
f_k(\gamma_k(t_{k+1}))\,.
\end{equation}
It follows from~(\ref{eq6C}) that, there is $1\leqslant
l=l(k)\leqslant m-1$ such that such that
\begin{equation}\label{eq7F}
h(f(\gamma_k(t_{l(k)})), f_k(\gamma_k(t_{l(k)+1}))\geqslant
\varepsilon_0/2l\geqslant \varepsilon_0/2(m-1)\,.
\end{equation}
Observe that, the set $G_k:=|\gamma_k|_{(t_{l(k)}, t_{l(k)+1})}|$
belongs to $D^{\,\prime}\setminus E.$

\medskip
Observe that, $G_k$ contains some a continuum $\widetilde{G}_k$ with
$h(\widetilde{G}_k)\geqslant \varepsilon_0/4(m-1)$ for any $k\in
{\Bbb N}.$ Indeed, by the construction of $G_k,$ there is a sequence
of points $a_{s, k}:=f_k(\gamma_k(p_s))\rightarrow
w_k:=f(\gamma_k(t_{l(k)}))$ as $s\rightarrow \infty$ and $b_{s,
k}=f_k(\gamma_k(q_s))\rightarrow f(\gamma_k(t_{l(k)+1}))$ as
$s\rightarrow \infty,$ where $t_{l(k)})<
p_s<q_s<\gamma_k(t_{l(k)+1}))$ and $a_{s, k}, b_{s, k}\in G_k.$
By the triangle inequality and by~(\ref{eq7F})
$$
h(a_{s, k}, b_{s, k})\geqslant h(b_{s, k}, w_k)-h(w_k, a_{s,
k})\geqslant
$$
\begin{equation}\label{eq11A}
h(f_k(\gamma_k(t_{l(k)+1})), w_k)-h(f_k(\gamma_k(t_{l(k)+1})), b_{s,
k})-h(w_k, a_{s, k})\geqslant
\end{equation}
$$\varepsilon_0/2(m-1)-h(f_k(\gamma_k(t_{l(k)+1}), b_{s,
k}))-h(w_k, a_{s, k})\,.$$
Since $a_{s, k}:=f_k(\gamma_k(p_s))\rightarrow
w_k:=f_k(\gamma_k(t_{l(k)}))$ as $s\rightarrow \infty$ and $b_{s,
k}=f_k(\gamma_k(q_s))\rightarrow f_k(\gamma_k(t_{l(k)+1}))$ as
$s\rightarrow \infty,$ it follows from the last inequality that
there is $s=s(k)\in {\Bbb N}$ such that
\begin{equation}\label{eq12B}
h(a_{s(k), k}, b_{s(k), k})\geqslant\varepsilon/4(m-1)\,.
\end{equation}
Now, we set
$$\widetilde{G}_k:=f_k(\gamma_k)|_{[p_{s(k)}, q_{s(k)}]}\,.$$
In other words, $\widetilde{G}_k$ is a part of the path
$f_k(\gamma_k)$ between points $a_{s(k), k}$ and $b_{s(k), k}.$
Since $\widetilde{G}_k$ is a continuum in
$D^{\,\prime}_{f_k}\setminus E_*,$ there is a component $K_k$ of
$D^{\,\prime}_{f_k}\setminus E_*,$ containing $K_k.$ Let us apply
the definition of equi-uniformity for he sets $\widetilde{G}_k$ and
$K_{f_k}$ in $K_k$ (here $K_{f_k}$ is a continuum from the
definition of the class $\frak{R}^{E_*, E, F}_{Q, \delta, p}(D)$, in
particular, $h(K_{f_k})\geqslant \delta.$) Due to this definition,
for the number $\delta_*:=\min\{\delta, \varepsilon/4(m-1)\}>0$
there is $P>0$ such that
\begin{equation}\label{eq1F}
M_p(\Gamma(\widetilde{G}_k, K_{f_k}, K_k))\geqslant P
>0\,, \qquad k=1,2,\ldots\,.
\end{equation}
Let us to show that, the relation~(\ref{eq1F}) contradicts with the
definition of $f$ in~(\ref{eq2*!A}) together with the
conditions~(\ref{eq8})--(\ref{eq9}). Indeed, let us denote by
$\Gamma_k$ the family of all half-open paths $\beta_k:[a,
b)\rightarrow \overline{{\Bbb R}^n}$ such that $\beta_k(a)\in
|f_k(\gamma_k)|,$ $\beta_k(t)\in K_k$ for all $t\in [a, b)$ and,
moreover, $\lim\limits_{t\rightarrow b-0}\beta_k(t):=B_k\in
K_{f_k}.$ Obviously, by~(\ref{eq1F})
\begin{equation}\label{eq4C}
M_p(\Gamma_k)=M_p(\Gamma(\widetilde{G}_k, K_{f_k}, K_k))\geqslant P
>0\,, \qquad k=1,2,\ldots\,.
\end{equation}
Consider the family $\Gamma_k^{\,\prime}$ of all maximal
$f_k$-liftings $\alpha_k:[a, c)\rightarrow D$ of the family
$\Gamma_k$ starting at $|\gamma_k|;$ such a family exists by
Proposition~\ref{pr3}.

Observe that, the situation when $\alpha_k\rightarrow \partial D$ as
$k\rightarrow\infty$ is impossible. Suppose the opposite: let
$\alpha_k(t)\rightarrow \partial D$ as $t\rightarrow c.$ Let us
choose an arbitrary sequence $\varphi_m\in [0, c)$ such that
$\varphi_m\rightarrow c-0$ as $m\rightarrow\infty.$ Since the space
$\overline{{\Bbb R}^n}$ is compact, the boundary $\partial D$ is
also compact as a closed subset of the compact space. Then there
exists $w_m\in
\partial D$ such that
\begin{equation}\label{eq7G}
h(\alpha_k(\varphi_m), \partial D)=h(\alpha_k(\varphi_m), w_m)
\rightarrow 0\,,\qquad m\rightarrow \infty\,.
\end{equation}
Due to the compactness of $\partial D$, we may assume that
$w_m\rightarrow w_0\in \partial D$ as $m\rightarrow\infty.$
Therefore, by the relation~(\ref{eq7G}) and by the triangle
inequality
\begin{equation}\label{eq8C}
h(\alpha_k(\varphi_m), w_0)\leqslant h(\alpha_k(\varphi_m),
w_m)+h(w_m, w_0)\rightarrow 0\,,\qquad m\rightarrow \infty\,.
\end{equation}
On the other hand,
\begin{equation}\label{eq9C}
f_k(\alpha_k(\varphi_m))=\beta_k(\varphi_m)\rightarrow \beta(c)
\,,\quad m\rightarrow\infty\,,
\end{equation}
because by the construction the path $\beta_k(t),$ $t\in [a, b],$
lies in $K_k\subset D^{\,\prime}_{f_k}\setminus E_*$ together with
its finite ones points. At the same time, by~(\ref{eq8C})
and~(\ref{eq9C}) we have that $\beta_k(c)\in C(f_k,
\partial D)\subset E_*$ by the definition of the class
$\frak{R}^{E_*, E, F}_{Q, \delta, p}(D).$ The inclusions
$\beta_k\subset D^{\,\prime}\setminus E_*$ and $\beta_k(c)\in E_*$
contradict each other.

\medskip
Therefore, by Proposition~\ref{pr3} $\alpha_k\rightarrow x_1\in D$
as $t\rightarrow c-0,$ and $c_b$ and $f_k(\alpha_k(b))=f_k(x_1).$ In
other words, the $f_k$-lifting $\alpha_k$ is complete, i.e.,
$\alpha_k:[a, b]\rightarrow D.$ Besides that, it follows from that
$\alpha_k(b)\in f_k^{\,-1}(K_{f_k}).$

\medskip
Again, by the definition of the class $\frak{R}^{E_*, E, F}_{Q,
\delta, p}(D),$
\begin{equation}\label{eq13B}
h(f_k^{\,-1}(K_{f_k}), \partial D)\geqslant \delta>0\,.
\end{equation}
Since $x_0\ne\infty,$ it follows from~(\ref{eq13B}) that
\begin{equation}\label{eq14C}
f_k^{\,-1}(K_{f_k})\subset D\setminus B(x_0, r_0)
\end{equation}
for any $k\in {\Bbb N}$ and some $r_0>0.$ Let $k$ be such that
$2^{\,-k}<\varepsilon_0.$ We may consider that $\varepsilon_0<r_0.$
Due to~(\ref{eq14C}), we may show that
\begin{equation}\label{eq15A}
\Gamma_k^{\,\prime}>\Gamma(S(x_0, 2^{\,-k}), S(x_0, \varepsilon_0),
D)
\end{equation}
(see \cite[Theorem~1.I.5, \S46]{Ku}). Observe that the function
$$\eta_k(t)=\left\{
\begin{array}{rr}
\psi(t)/I(2^{\,-k}, \varepsilon_0), &   t\in (2^{\,-k},
\varepsilon_0),\\
0,  &  t\in {\Bbb R}\setminus (2^{\,-k}, \varepsilon_0)\,,
\end{array}
\right. $$
where $I(2^{\,-k}, \varepsilon_0)$ is defined in~(\ref{eq7***}),
satisfies~(\ref{eqA2}) for $r_1:=2^{\,-k},$ $r_2:=\varepsilon_0.$
Therefore, by the definition of the mapping in~(\ref{eq2*!A}) and
by~(\ref{eq15A}) we obtain that
\begin{equation}\label{eq11*A}
M_p(f(\Gamma^{\,\prime}_k))\leqslant \Delta(k)\,,
\end{equation}
where $\Delta(k)\rightarrow 0$ as $k\rightarrow \infty.$ However,
$\Gamma_k=f_k(\Gamma_k^{\,\prime}).$ Thus, by~(\ref{eq11*A}) we
obtain that
\begin{equation}\label{eq3A}
M_p(\Gamma_k)= M_p(f_k(\Gamma_k^{\,\prime}))\leqslant
\Delta(k)\rightarrow 0
\end{equation}
as $k\rightarrow \infty.$ However, the relation~(\ref{eq3A})
together with the inequality~(\ref{eq4C}) contradict each other,
which proves the lemma. $\Box$
\end{proof}

\medskip
{\bf \noindent Victoria Desyatka} \\
Zhytomyr Ivan Franko State University,  \\
40 Velyka Berdychivs'ka Str., 10 008  Zhytomyr, UKRAINE \\
victoriazehrer@gmail.com

\medskip
\medskip
{\bf \noindent Evgeny Sevost'yanov} \\
{\bf 1.} Zhytomyr Ivan Franko State University,  \\
40 Velyka Berdychivs'ka Str., 10 008  Zhytomyr, UKRAINE \\
{\bf 2.} Institute of Applied Mathematics and Mechanics\\
of NAS of Ukraine, \\
19 Henerala Batyuka Str., 84 116 Slov'yansk,  UKRAINE\\
esevostyanov2009@gmail.com

\end{document}